\documentclass[preprint,12pt]{elsarticle}
\usepackage{epsf}
\usepackage{amssymb,latexsym,amsmath,amscd,array,graphicx,amsthm}
\usepackage{mathtools}
\usepackage{hyperref}
\hypersetup{hidelinks}

\swapnumbers
\numberwithin{equation}{section}

\makeatletter
\def\ps@pprintTitle{%
   \let\@oddhead\@empty
   \let\@evenhead\@empty
   \def\@oddfoot{\reset@font\hfil}%
   \def\@evenfoot{\reset@font\hfil}%
}
\makeatother

\newcommand{\R}{\mathbb{R}}

\newcommand{\N}{\mathbb{N}}

\newcommand{\im}{\mathrm{im}}
\newcommand{\codim}{\mathrm{codim}}
\newcommand{\ind}{\mathrm{ind}}

\newcommand{\interior}{\mathrm{int}}

\newcommand{\sky}{\mathrm{sky}}
\newcommand{\grad}{\mathrm{grad}}

\theoremstyle{plain}
\newtheorem{thm}{Theorem}[section]
\newtheorem{cor}[thm]{Corollary}
\newtheorem{lem}[thm]{Lemma}
\newtheorem{prop}[thm]{Proposition}

\newtheorem{result}[thm]{Result}

\theoremstyle{definition}
\newtheorem{defin}[thm]{Definition}
\newtheorem{definthm}[thm]{Definition-Theorem}

\newtheorem{rem}[thm]{Remark}
\newtheorem{example}[thm]{Example}

\newtheorem{question}[thm]{Question}

\newcommand\theoref{Theorem~\ref}

\newcommand\lemref{Lemma~\ref}
\newcommand\remref{Remark~\ref}
\newcommand\propref{Proposition~\ref}
\newcommand\corolref{Corollary~\ref}
\newcommand\defref{Definition~\ref}

\newcommand\questionref{Question~\ref}
\newcommand\secref{Section~\ref}
\newcommand\subsecref{Subsection~\ref}

\begin{document}

\begin{frontmatter}

\title{Refocusing spacetimes need not be strongly refocusing}
\author[inst1]{Friedrich Bauermeister\fnref{orcid1}}
\ead{friedrich.bauermeister.gr@dartmouth.edu}
\fntext[orcid1]{ORCID: \href{https://orcid.org/}{0009-0006-6513-2153}}
\address[inst1]{Department of Mathematics, Dartmouth College\\
Hanover, New Hampshire, United States}

\begin{abstract}
We prove that there are globally hyperbolic spacetimes which are refocusing but not strongly refocusing. In fact, every globally hyperbolic strongly refocusing spacetime of dimension at least $3$ admits globally hyperbolic metrics which are refocusing but not strongly refocusing. This answers a question by Chernov, Kinlaw, and Sadykov. We then prove that globally hyperbolic spacetimes which are Legendrian refocusing (a notion introduced in this paper) admit globally hyperbolic strongly refocusing metrics. As a corollary, a contact Bott-Samelson type result by Frauenfelder, Labrousse, and Schlenk can be applied to Legendrian refocusing spacetimes to show that the Cauchy surface of a globally hyperbolic Legendrian refocusing spacetime of dimension at least $3$ is compact, that its fundamental group is finite, and that its universal cover has the integral cohomology ring of a compact rank one symmetric space (CROSS).
\end{abstract}
\end{frontmatter}

\section{Introduction}

Unless otherwise stated, any manifold in this paper is smooth and without boundary, and geodesics of Riemannian and Lorentzian manifolds are assumed to be maximal.

\begin{defin}[strongly refocusing spacetime]\label{def:strongly refocusing}
Let $(X,g)$ be a spacetime and let $p,q\in X$ with $p\neq q$. We say that $(X,g)$ is \textit{strongly refocusing with respect to $p$ and $q$} if every null-geodesic through $p$ also passes through $q$. In this case we also have that all null-geodesics through $q$ go through $p$, see \cite[Proposition~3.3]{Kinlaw2011}. We say that $(X,g)$ is \textit{strongly refocusing} if there exist $p\neq q$ in $X$ such that $(X,g)$ is strongly refocusing with respect to $p$ and $q$. 
\end{defin}

A weaker notion was introduced earlier by Low \cite{Low1993, Low2001, Low2006}, who gave three different definitions. A proof that the three different definitions are equivalent was later published by Kinlaw \cite[Proposition~3.7]{Kinlaw2011}.

\begin{definthm}[refocusing spacetime]\label{def:equivalent definitions of refocusing}
Let $(X,g)$ be a strongly causal spacetime. The following three conditions, corresponding to Low's three definitions of refocusing, are equivalent:
\begin{itemize}
\item[(i)] There exist $p\in X$ and an open $U\ni p$ such that for every open $V$ with $p\in V\subseteq U$ there is $q\in X\setminus V$ with every null-geodesic through $q$ entering $V$.
\item[(ii)] There exist $p\in X$ and an open $U\ni p$ such that for every open $V$ with $p\in V\subseteq U$ there is $q\in X\setminus U$ with every null-geodesic through $q$ entering $V$.
\item[(iii)] There exist $p\in X$ and a sequence $(q_n)_{n\in\N}$ which does not have $p$ as an accumulation point such that for every open neighborhood $V$ of $p$ there is some $N\in\N$ such that every null-geodesic through $q_n$ goes through $V$ for $n\geq N$. In this case we say that $(X,g)$ is \textit{refocusing with respect to the point $p\in X$ and the sequence $(q_n)_{n\in \N}$ in $X$}.
\end{itemize}
We say that $(X,g)$ is a refocusing spacetime if any, hence every, of these three equivalent conditions is fulfilled. The implications (iii)$\Leftrightarrow$(ii)$\Rightarrow$(i) hold for arbitrary spacetimes. We use formulation (iii) as our working definition throughout.
\end{definthm}
Every strongly refocusing spacetime is refocusing.
The following question is due to Chernov, Kinlaw, and Sadykov \cite{ChernovKinlawSadykov2010}.
\begin{question}[{\cite{ChernovKinlawSadykov2010}}]\label{quest:Are there refocusing non-strongly refocusing spacetimes?}
Are there globally hyperbolic refocusing spacetimes which are not strongly refocusing?
\end{question}

A weaker form of \questionref{quest:Are there refocusing non-strongly refocusing spacetimes?} was answered affirmatively by Kinlaw \cite{Kinlaw2011}, who exhibited a globally hyperbolic spacetime that is refocusing at a point but fails to be strongly refocusing at that same point.

\begin{example}[{\cite{Kinlaw2011}}]\label{ex:S2Rinterval}
Let $h$ be the round metric on $S^2$ and let $(X,g)=(S^2\times(-\pi,\pi),\,h-dt^2)$. Fix $x\in S^2$ and let $-x$ be its antipodal point. The sequence $q_n=(-x,\pi-\tfrac{1}{n})$ shows that $(X,g)$ is refocusing with respect to $p=(x,0)$, but it is not strongly refocusing with respect to $p$. Note however that $(X,g)$ \emph{is} strongly refocusing at other points.
\end{example}

We resolve \questionref{quest:Are there refocusing non-strongly refocusing spacetimes?} in the affirmative, proving the following, stronger theorem.

\begin{result}[\theoref{thm:main}]
Every globally hyperbolic strongly refocusing spacetime $(X,g)$ of dimension at least $3$ admits a globally hyperbolic metric $g'$ on $X$ which is refocusing but not strongly refocusing.
\end{result}

To place this in the context of Low's reconstruction program, let $(X,g)$ be a globally hyperbolic spacetime, let $\mathcal{N}$ denote its space of lightrays, and let $\Sigma$ denote its set of skies. For an event $p\in X$, its \textit{sky} is the subset $\sky_p\coloneqq \{[\alpha]\in\mathcal{N}\mid \alpha \text{ is a null geodesic through } p\}.$ A sensible thing to do is to endow $\Sigma$ with a topology. Low proposes the following choice of topology, which he calls the \textit{reconstructive topology}.

\begin{defin}[coarse topology on the space of skies]
Let $(X,g)$ be a globally hyperbolic spacetime, let $\mathcal{N}$ be the space of lightrays of $(X,g)$ and let $\Sigma$ be the space of skies. Consider $\Sigma$ as a subspace $\mathcal{C}(\mathcal{N})\coloneqq \{S \mid S \subseteq \mathcal{N} \text{ closed} \}$ equipped with the upper Vietoris topology, i.e. the topology generated by sets of the form $\{S_1 \in \mathcal{C}(\mathcal{N}) \mid S_1\subseteq U_{\mathcal{N}}\}$ where $U_{\mathcal{N}}\subseteq \mathcal{N}$ is open. We call this topology on $\Sigma$ the \textit{coarse topology} on $\Sigma$. We denote the set $\Sigma$ equipped with this topology by $\Sigma_c$.
\end{defin}

\begin{lem}
Let $(X,g)$ be a globally hyperbolic spacetime. Let $p\in X$. Then sets of the form $U_{\Sigma}\coloneqq \{\sky\in \Sigma_c \mid \alpha \text{ intersects } U \text{ for all } \alpha \in \sky\},$ where $U$ is an open neighborhood of $p$, form a neighborhood basis of $\sky_p$ in $\Sigma_c$.
\end{lem}
\begin{proof}
The fact that $U_{\Sigma}\subseteq \Sigma_c$ is open follows from the fact that $\{\alpha \in \mathcal{N} \mid \alpha \text{ intersects } U\}\subseteq \mathcal{N}$ is open. Let $V\subseteq \mathcal{N}$ be an open set and let $\sky_p \subseteq V$. For any open neighborhood $U\subseteq X$ of $p$ we have that $\sky_p \in U_{\Sigma}$. We have to show that for sufficiently small $U$ we also have $\sky \in U_{\Sigma} \implies \sky \subseteq V$. Assume not. Since $X$ is a manifold, we can choose a neighborhood basis $(U_n)_{n\in \N}$ of $p$ with $U_{n+1}\subseteq U_n$ and $\overline{U_n}$ compact for all $n\in \N$. By assumption, there is then a sequence $(\sky_n)_{n\in \N}$ with $\sky_n \in (U_n)_{\Sigma}$ such that for each $n\in \N$ there exists some $\alpha_n \in \sky_n$ with $\alpha_n \notin V$. Since $(X,g)$ is globally hyperbolic, the set of lightrays intersecting the compact set $\overline{U_1}$ is compact in $\mathcal{N}$, so we can take a convergent subsequence $(\alpha_{n_k})_{k\in \N}$. Since $V$ is open and since $\alpha_n \notin V$ for all $n\in \N$ we conclude that $\alpha \coloneqq \lim_{k\to\infty} \alpha_{n_k} \notin V$. On the other hand, since $\{p\} = \bigcap_{n\in \N} U_n$, we have $\alpha \in \sky_p \subseteq V$, so we have arrived at a contradiction.
\end{proof}

\begin{prop}
Let $(X,g)$ be a globally hyperbolic spacetime and let $\Sigma$ be its set of skies. The map $S: X\to \Sigma_c$ defined by $p\to \sky_p$ is continuous and it is a homeomorphism if and only if $(X,g)$ is non-refocusing.
\end{prop}
\begin{proof}
The fact that $(X,g)$ being non-refocusing implies that the map $X\to \Sigma_c$ is a homeomorphism is first stated by Low in \cite{Low2001} and proved by Kinlaw in \cite{Kinlaw2011}. The converse direction follows from the fact that if $(X,g)$ is refocusing w.r.t. $p$ and $\{q_n\}_{n\in \N}$, then $q_n$ does not converge to $p$ but $\sky_{q_n}\to \sky_p$ in $\Sigma_c$.
\end{proof}

In many ways the most natural way to view $\Sigma$ as a subspace of a larger space is to embed $\Sigma$ into the space $\mathcal{L}$ of smooth Legendrian submanifolds of $\mathcal{N}$ which are Legendrian isotopic to one (hence every) sky. This is motivated for example by the Legendrian Low conjecture, which states that points in a globally hyperbolic spacetime $(X,g)$ are causally related if and only if their skies are Legendrian linked in $\mathcal{N}$. The Legendrian Low conjecture was proved for a broad class of spacetimes by Chernov and Nemirovski \cite{ChernovNemirovski2010}. Strong refocusing is an obstruction to the Legendrian Low conjecture.

We equip $\mathcal{L}$ with the $\mathcal{C}^{\infty}$ topology, and use this to define a finer topology on the space of skies.

\begin{defin}[fine topology on the space of skies]
Let $(X,g)$ be a globally hyperbolic spacetime, let $\mathcal{N}$ be the space of lightrays of $(X,g)$ and let $\Sigma$ be the space of skies. Let $\mathcal{L}$ be the space of smooth Legendrian submanifolds of $\mathcal{N}$ Legendrian isotopic to a sky, equipped with the $\mathcal{C}^{\infty}$ topology. We call the subspace topology on $\Sigma$ as a subset of $\mathcal{L}$ the \textit{fine topology}. We denote the set $\Sigma$ equipped with this topology by $\Sigma_f$.
\end{defin}

\begin{defin}
Let $(X,g)$ be a globally hyperbolic spacetime. We say that $(X,g)$ is \textit{Legendrian refocusing} with respect to a point $p\in X$ and a sequence $(q_n)_{n\in \N}$ in $X$ if $p$ is not an accumulation point of $\{q_n\}_{n\in \N}$ but $\sky_{q_n} \to \sky_{p}$ in $\mathcal{L}$. We say that $(X,g)$ is \textit{Legendrian refocusing} if it is Legendrian refocusing with respect to some point $p\in X$ and some sequence $(q_n)_{n\in \N}$ in $X$.
\end{defin}

The condition of a spacetime being Legendrian refocusing sits between being strongly refocusing and being refocusing. The benefit of this definition is that it fits nicely into the Legendrian picture.

\begin{rem}
Let $(X,g)$ be a globally hyperbolic spacetime. Then $(X,g)$ is Legendrian refocusing if and only if the sky map $S:X\to\Sigma_f$ fails to be a homeomorphism, equivalently if and only if the sky map $S:X\to \mathcal{L}$ fails to be an embedding.
\end{rem}

The metrics constructed in \theoref{thm:main} are examples of globally hyperbolic spacetimes which are Legendrian refocusing but not strongly refocusing.

We also prove the following.

\begin{result}[\theoref{thm:Legendrian refocusing implies existence of strongly refocusing metric}]
Let $(X,g)$ be a globally hyperbolic spacetime which is Legendrian refocusing. Then $X$ admits a globally hyperbolic metric $g'$ which is strongly refocusing.
\end{result}

\begin{rem}\label{rem:BautistaIbortLafuente}
In \cite{BautistaIbortLafuente2015}, Bautista, Ibort, and Lafuente treat essentially the same question as \questionref{quest:Are there refocusing non-strongly refocusing spacetimes?}. They define the notion of a spacetime being sky-separating, which in our language corresponds to a spacetime not being strongly refocusing. They claim in \cite[Corollary 4.8]{BautistaIbortLafuente2015} that sky-separating globally hyperbolic spacetimes are non-refocusing. The examples constructed in this paper show that this statement is false.
\end{rem}

\begin{rem}
Throughout this paper, whenever we say that a globally hyperbolic spacetime $(X,g)$ admits a globally hyperbolic metric $g'$ fulfilling some property, the reader is welcome to conclude that $(X,g')$ and $(X,g)$ have diffeomorphic Cauchy surfaces. A priori this assumption is not justified; there certainly are diffeomorphic globally hyperbolic spacetimes with non-diffeomorphic Cauchy surfaces. But for our constructions it is true.
\end{rem}

The paper is organized as follows. In \secref{sec:preliminaries} we provide preliminaries, including the definition of globally hyperbolic spacetimes and a discussion of the space of lightrays as a contact manifold. In \secref{sec:Legendrian refocusing spacetimes need not be strongly refocusing} we prove \theoref{thm:main}, namely that every globally hyperbolic strongly refocusing spacetime of dimension at least $3$ admits a globally hyperbolic Legendrian refocusing, non-strongly refocusing metric, using the Sard--Smale transversality theorem for Banach manifolds. In \secref{sec:Legendrian refocusing spacetimes admit strongly refocusing metrics} we prove \theoref{thm:Legendrian refocusing implies existence of strongly refocusing metric}, which states that every globally hyperbolic Legendrian refocusing spacetime admits a strongly refocusing metric. In Appendix~\ref{appendix: Topological consequences of refocusing and strong refocusing} we include a discussion of topological consequences of being refocusing and strongly refocusing, as well as a discussion of Riemannian analogues of (strong) refocusing. Since the manifolds admitting globally hyperbolic Legendrian refocusing metrics are proven to be exactly those admitting globally hyperboli  strongly refocusing metrics, being Legendrian refocusing carries the same topological consequences as being strongly refocusing.

\begin{result}[\corolref{cor:Legendrian refocusing implies integral cohomology of a CROSS}]
Let $(X,g)$ be a globally hyperbolic spacetime of dimension at least $3$ which is Legendrian refocusing. Then any Cauchy surface of $(X,g)$ is compact with finite fundamental group, and its universal cover has the integral cohomology ring of a compact rank one symmetric space (CROSS).
\end{result}

\section{Preliminaries}\label{sec:preliminaries}
We quickly review the definition and some key properties of globally hyperbolic spacetimes.

\begin{defin}[globally hyperbolic spacetime, Cauchy surface]\label{def:gh}
A spacetime $(X,g)$ is \textit{causal} if it contains no closed causal curve, and \textit{globally hyperbolic} if it is causal and if $J^+(p)\cap J^-(q)$ is compact for all $p,q\in X$. A subset $M\subset X$ is a \textit{topological Cauchy surface} of $(X,g)$ if it is a topological hypersurface meeting every inextendible timelike curve in exactly one point, and an \textit{acausal Cauchy surface} if it meets every inextendible causal curve in exactly one point.
\end{defin}

A spacetime $(X,g)$ is globally hyperbolic if and only if it admits an acausal topological Cauchy surface $M$. In that case any two topological acausal Cauchy surfaces are homeomorphic and $X$ is homeomorphic to $M\times\R$; see Geroch \cite{Geroch1970} and \cite{HawkingEllis}. Smooth refinements of all the above are due to Bernal and S\'anchez. In particular: $(X,g)$ is globally hyperbolic if and only if it admits a smooth, spacelike Cauchy surface $M$; any two smooth Cauchy surfaces of $(X,g)$ are diffeomorphic; and $X$ is diffeomorphic to $M\times\R$ \cite{BernalSanchez2003, BernalSanchez2005}. Until Bernal and S\'anchez \cite{BernalSanchez2007} showed that globally hyperbolic spacetimes can be defined as causal spacetimes with compact causal diamonds, they were defined by requiring them to be strongly causal. Globally hyperbolic spacetimes are the ones compatible with Penrose's strong cosmic censorship hypothesis \cite{Penrose1998}, and are generally regarded as the physically reasonable spacetimes.

\begin{defin}[temporal functions]
For a spacetime $(X,g)$, a smooth function 
$\mathcal{T}:X\to \R$ is called a \textit{temporal function} if its gradient is a past-pointing timelike vectorfield. A surjective temporal function $\mathcal{T}:X\to \R$ is called \textit{Cauchy} if its level sets are Cauchy surfaces, i.e. if for all $t\in \R$ we have that $\mathcal{T}^{-1}(t)\subset X$ is a Cauchy surface. Every globally hyperbolic spacetime $(X,g)$ admits a surjective Cauchy temporal function $\mathcal{T}:X\to \R$. This was proved by Bernal and S\'anchez \cite{BernalSanchez2003, BernalSanchez2004, BernalSanchez2005}.
\end{defin}
Temporal functions are strictly increasing along future-pointing causal curves.

Low showed in \cite{Low2006} that the set $\mathcal{N}$ of future-directed, unparametrized, inextendible null geodesics of a strongly causal spacetime $(X,g)$ carries a natural smooth structure and a natural contact structure. When $(X,g)$ is globally hyperbolic, this contact manifold can be identified with the spherical cotangent bundle of any spacelike Cauchy surface. To see this, let $M\subset X$ be a spacelike Cauchy surface. Every future-directed null geodesic $\alpha$ meets $M$ in a unique point, say $\alpha(t)=p$. Let $\pi:T_pX\to T_pM$ be the orthogonal projection and set $v=\pi(\alpha'(t))\in T_pM$. Note that $\ker(\pi)\subseteq T_pX$ is timelike and since $\alpha'(t)$ is null we have that $v\neq 0$. Changing the affine parametrization of $\alpha$ by an orientation-preserving reparametrization changes $v$ by multiplication with a positive scalar. Thus the assignment $[\alpha]\mapsto [v]$ defines a map from $\mathcal{N}$ to the spherical tangent bundle $STM$. This map is a diffeomorphism. Using the Riemannian metric induced by $g|_M$ on $M$, we identify $STM$ with $ST^{\ast}M$, and hence obtain a diffeomorphism $\mathcal{N}\simeq ST^{\ast}M$. One checks that this is a contactomorphism. For an event $p\in X$, its \textit{sky} is the subset $\sky_p\coloneqq \{[\alpha]\in\mathcal{N}\mid \alpha \text{ is a null geodesic through } p\}.$ Each sky is a Legendrian submanifold of $\mathcal{N}$. Moreover, if the chosen Cauchy surface $M$ contains $p$, then under the identification $\mathcal{N}\simeq ST^{\ast}M$, $\sky_p$ is exactly the fiber $ST_p^{\ast}M$. Let $(X,g)$ be a globally hyperbolic spacetime, and let $\Sigma$ denote its set of skies. Then Low's program asks us to consider in what way we can reconstruct $(X,g)$ from $\Sigma$.

\section{Legendrian refocusing spacetimes need not be strongly refocusing}\label{sec:Legendrian refocusing spacetimes need not be strongly refocusing}
In this section we prove that Legendrian refocusing globally hyperbolic spacetimes need not be strongly refocusing. We first explain the strategy. Start with a globally hyperbolic spacetime $(X,g)$ which is strongly refocusing with respect to points \(p,q\), with
$q\in J^+(p)$. After choosing a spacelike Cauchy surface $M$ through $q$ and setting $X'=I^-(M)$, the restricted spacetime $(X',g|_{X'})$ is Legendrian refocusing with respect to $p$ and a sequence $(q_n)_{n\in\N}\subset X'$ which converges to $q$ in $X$. The main idea is to perturb the metric on $X'$ so as to destroy all
strong refocusing, while keeping the metric fixed on a regular closed set $A$ containing the null-geodesic data responsible for this Legendrian refocusing. We will use the Sard--Smale transversality theorem to show that a residual set of such perturbations destroys all strong refocusing not already contained in $A$. A minimality argument is then used to choose $A$ so that no strong refocusing is contained in $A$ in the first place.

\subsection{A Banach manifold of Lorentzian metrics and refocusing within subsets}\label{subsec:banach}

\begin{defin}
Let $X$ be a smooth manifold. We say that $A\subseteq X$ is a \textit{regular closed} set if it is the closure of its interior. We say that $A\subseteq X$ is a \textit{regular open} set if it is the interior of its closure.
\end{defin}
A set is regular closed if and only if its complement is regular open. We will use the following fact: If two $\mathcal{C}^k$ tensorfields $g_1$ and $g_2$ on a manifold $X$ agree on a regular closed set $A\subseteq X$, then they agree on $A$ to $k$-th order. This is true because if they agree on $A$, they agree on $\interior(A)$, on which they agree to $k$-th order. Since $A = \overline{\interior(A)}$ and since $g_1$ and $g_2$ are $\mathcal{C}^k$, we get that $g_1$ and $g_2$ agree to $k$-th order on $A$ by continuity.

\begin{defin}\label{def: Giving a subset of Lorentz metrics the structure of a Banach manifold}
Let $k\geq 0$. Let $(X,g)$ be a globally hyperbolic spacetime of dimension $n$, with $g$ a $\mathcal{C}^k$ metric. Fix a smooth auxiliary Riemannian metric $e$ on $X$ with Levi-Civita connection $\nabla$. Since global hyperbolicity is an open condition in Whitney's fine $\mathcal{C}^0$ topology \cite{NavarroMinguzzi2011}, there exists a continuous function $\varepsilon: X \to (0,\infty)$ such that any symmetric $2$-tensor $h$ satisfying $|h - g|_e < \varepsilon$ is a globally hyperbolic Lorentzian metric on $X$. Choose a smooth function $W: X \to (0,\infty)$ satisfying $W < \varepsilon$. We define $\Gamma^k_{W,0}(S^2T^{\ast}X)$ as the vectorspace of $\mathcal{C}^k$ symmetric $2$-tensors for which
$$\|u\|_{\mathcal{C}^k_W} \coloneqq \max_{0 \leq j \leq k} \sup_{x \in X} \frac{|\nabla^j u(x)|_e}{W(x)}$$
is finite and which vanish at infinity in the sense that
$$\lim_{x\to\infty}\max_{0\leq j\leq k}\frac{|\nabla^j u(x)|_e}{W(x)}=0.$$
Here $x\to\infty$ means that for every $\delta>0$ there exists a compact set $K\subset X$ such that
$$\max_{0\leq j\leq k}\sup_{x\in X\setminus K}\frac{|\nabla^j u(x)|_e}{W(x)}<\delta.$$
The space $\Gamma^k_{W,0}(S^2T^{\ast}X)$ equipped with the norm $\|u\|_{\mathcal{C}^k_W}$ is a separable Banach space. Let
$$\mathcal{U}^k \coloneqq \left\{u\in \Gamma^k_{W,0}(S^2T^*X)\ \middle|\ \|u\|_{\mathcal{C}^0_W}<1\right\}.$$
This is an open subset of $\Gamma^k_{W,0}(S^2T^*X)$. We define the space of globally hyperbolic metrics near $g$ as
$$\mathcal{M}^k \coloneqq g+\mathcal{U}^k.$$
The space $\mathcal{M}^k$ thus carries the structure of a separable Banach manifold. Throughout the rest of this paper we will use ``Banach manifold'' to mean ``separable Banach manifold''. By construction, every $h \in \mathcal{M}^k$ satisfies $|h - g|_e < W < \varepsilon$ everywhere, ensuring that every metric in $\mathcal{M}^k$ is a globally hyperbolic Lorentzian metric. For a regular closed subset $A \subseteq X$, let $\mathcal{Z}^k_A \subseteq \Gamma^k_{W,0}(S^2 T^*X)$ be the closed linear subspace of tensors which vanishes on $A$, i.e., $u(x)=0 $ for all $x\in A$. We define
$$\mathcal{M}^k_{(A,g)} \coloneqq g + (\mathcal{U}^k \cap \mathcal{Z}^k_A),$$
the space of metrics near $g$ which agree with $g$ on $A$. Note that since we assume $A\subseteq X$ to be a regular closed set, metrics in $\mathcal{M}^k_{(A,g)}$ actually agree with $g$ to order $k$ on $A$. Since $\mathcal{U}^k \cap \mathcal{Z}^k_A$ is an open subset of the Banach space $\mathcal{Z}^k_A$, the space $\mathcal{M}^k_{(A,g)}$ is also a Banach manifold. Finally, the smooth limits are defined as
$$\mathcal{M}^{\infty} \coloneqq \bigcap_{k \geq 0} \mathcal{M}^k, \quad \mathcal{M}^{\infty}_{(A,g)} \coloneqq \bigcap_{k \geq 0} \mathcal{M}^k_{(A,g)}.$$
Equipped with the inverse limit topologies, $\mathcal{M}^{\infty}$ and $\mathcal{M}^{\infty}_{(A,g)}$ are Fr\'echet manifolds. 
\end{defin}

\begin{defin}
Let $k\geq 2$, let $A\subseteq X$ be regular closed, let $g'\in \mathcal{M}^k_{(A,g)}$, and let $p,q\in X$, $p\neq q$. We say that $p$ and $q$ are \textit{$g'$-null-geodesically connected through} $X\setminus A$ if there exists a $g'$-null-geodesic $\alpha:[0,1]\to X$ such that $\alpha(0)=p$, $\alpha(1)=q$, and $\alpha([0,1]) \cap (X\setminus A) \neq \emptyset$. We say they are \textit{$N$-fold $g'$-null-geodesically connected through} $X\setminus A$ if there exist $N$ distinct such $g'$-null-geodesics. If this holds for some pair of distinct points $p,q \in X$, we say $g'$ is \textit{$N$-fold null-geodesically connecting through $X\setminus A$}.
\end{defin}

\begin{defin}
Let $(X,g)$ be a globally hyperbolic spacetime, let $p\in X$ and let $S\subseteq X$ be any subset. We define 
$$C_g(p,S)\coloneqq\{\alpha(t) \mid \alpha \text{ is a } g\text{-null-geodesic},\; \alpha(0)=p,\; \alpha(1) \in S,\; t\in [0,1]\}.$$
If $S=\{q\}$ we simply write $C_g(p,q)$.
\end{defin}

\begin{defin}
Let $(X,g)$ be a spacetime and let $A\subseteq X$ be any subset. We say $(X,g)$ is \textit{strongly refocusing within} $A$ if there exist $p,q\in A$, $p\neq q$ such that $(X,g)$ is strongly refocusing with respect to $p$ and $q$ and such that $C_g(p,q)\subseteq A.$ We say that $(X,g)$ is \textit{(Legendrian) refocusing within $A$} if there exists a point $p\in A$, a smooth spacelike Cauchy surface $M\subseteq X$ containing $p$, a neighborhood $U\subseteq X$ of $p$ and a sequence $(q_n)_{n\in \N}$ in $A\setminus U$ such that
\begin{itemize}
\item $(X,g)$ is (Legendrian) refocusing with respect to the point $p\in X$ and the sequence $(q_n)_{n\in \N}$
\item $C_g(q_n,M)\subseteq A$ for all $n\in \N$.

\end{itemize}
In that case we say that $(X,g)$ is (Legendrian) refocusing within $A$ with respect to the point $p$ and the sequence $(q_n)_{n\in \N}$.
\end{defin}
We note that being strongly refocusing within a regular closed subset is invariant under changes of the metric away from the subset. The following remark makes this precise. 
\begin{rem}\label{rem: strongly refocusing within a subset is stable}
Let $(X,g)$ be a globally hyperbolic spacetime, $A\subseteq X$ a regular closed set. Let $g'\in \mathcal{M}_{(A,g)}^k$ for $k\geq 2$. Then $(X,g')$ is strongly refocusing within $A$ if and only if $(X,g)$ is strongly refocusing within $A$.
\end{rem}
Refocusing and Legendrian refocusing within a regular closed subset are also stable under such changes of the metric. We prove the Legendrian case, which is the only one we use; the proof for refocusing proceeds along similar lines.
\begin{prop}\label{prop: Legendrian refocusing within a subset is stable}
Let $(X,g)$ be a globally hyperbolic spacetime, $A\subseteq X$ a regular closed set. Let $g'\in \mathcal{M}_{(A,g)}^{\infty}$. Then $(X,g')$ is Legendrian refocusing within $A$ if and only if $(X,g)$ is Legendrian refocusing within $A$.
\end{prop}
\begin{proof}
Let $X$ be a smooth connected manifold, let $A\subseteq X$ be regular closed and let $g,g'$ be two globally hyperbolic metrics on $X$ so that $g$ and $g'$ agree on $A$. Since $A$ is regular closed, $g$ and $g'$ then agree to infinite order on $A$. We will show that if one of $(X,g)$ and $(X,g')$ is Legendrian refocusing within $A$, then so is the other. Assume that $(X,g)$ is Legendrian refocusing within $A$, say with respect to a point $p\in A$ and a sequence $(q_n)_{n \in \N}$ in $A$. Let $M$ be the Cauchy surface through $p$ such that $C_g(q_n,M)\subseteq A$ for all $n\in \N$. Since $g$ and $g'$ and their derivatives agree on $A$, and since $p\in A$, there exists a small open subset $m\subseteq M$ containing $p$ such that the closure $\overline{m}$ of $m$ in $M$ is compact as well as acausal and spacelike with respect to $g'$. Since $\overline{m}$ is compact as well as spacelike and acausal with respect to $g'$, by the Bernal--S\'anchez extension theorem \cite[Theorem~1.1]{BernalSanchez2006}, $\overline{m}$ is contained in a smooth spacelike Cauchy surface $M'$ of $(X,g')$. Let $\mathcal{L}^g_m$ be the space of Legendrian submanifolds of $\mathcal{N}_g$ which are made up of null-geodesics through $m$ and let $\mathcal{L}^{g'}_m$ be the space of Legendrian submanifolds in $\mathcal{N}_{g'}$ which are made up of null-geodesics through $m$. Note that we can identify both $\mathcal{L}^g_m$ and $\mathcal{L}^{g'}_m$ with the space $\mathcal{L}_m$ of Legendrian submanifolds of $ST^{\ast}m$. We denote the identifications by $F_g:\mathcal{L}^g_m \to \mathcal{L}_m$ and $F_{g'}:\mathcal{L}^{g'}_m \to \mathcal{L}_m$ respectively. For sufficiently large $n$ we have that $\sky^g_{q_n} \in \mathcal{L}^g_m$ and, since $m\subseteq M$, that
$$
C_g(q_n,m)=C_g(q_n,M)\subseteq A.
$$
We claim that, for sufficiently large $n$,
$$
C_{g'}(q_n,M')=C_{g'}(q_n,m)=C_g(q_n,m)\subseteq A.
$$
For sufficiently large $n$ every $g$-null-geodesic through $q_n$ intersects $m$. Note that $g$ and $g'$ agree with each other on $A$, hence in particular on $C_g(q_n,m)$, to infinite order. At $q_n$ the $g$- and $g'$-null cones agree. For each null initial direction
at $q_n$, along the corresponding $g$-geodesic segment from $q_n$ to $m$, the metrics $g$ and $g'$ have the same Christoffel symbols. Uniqueness of solutions for the geodesic equation implies that the $g'$-geodesic with the same initial data is the same segment, so we have that $C_{g'}(q_n,m)=C_g(q_n,m)$. Further, $C_{g}(q_n,m)$ contains null-geodesic segments corresponding to every null-direction from $q_n$, the same is true for $C_{g'}(q_n,m)$. Hence $C_{g'}(q_n,M')=C_{g'}(q_n,m)$. This proves the claim. For those sufficiently large $n$ we have
$$
F_g(\sky^{g}_{q_n}) = F_{g'}(\sky^{g'}_{q_n}).
$$
Therefore
\begin{align*}
\lim_{n\to \infty} \sky^{g'}_{q_n} &= \lim_{n\to \infty} F_{g'}^{-1}(F_g(\sky^g_{q_n}))\\
&= F_{g'}^{-1}(F_g(\sky^g_p))\\
&= F_{g'}^{-1}(ST^{\ast}_p m)\\
&= \sky^{g'}_p.
\end{align*}
This proves that $(X,g')$ is Legendrian refocusing within $A$ with respect to $p$ and the sequence $(q_n)_{n\geq N}$ for some $N\in \N$.
\end{proof}

\subsection{Generic metrics are not strongly refocusing}\label{subsec:generic metrics}

We continue with the notation of \subsecref{subsec:banach}: $(X,g)$ is a globally hyperbolic spacetime, $e$ is an auxiliary Riemannian metric, $W$ is a weight function, and for any regular closed $A\subseteq X$ the space $\mathcal{M}^k_{(A,g)}$ denotes the Banach manifold of globally hyperbolic metrics agreeing with $g$ on $A$ (and hence agreeing to $k$-th order).

\begin{lem}\label{lem:Geodesics can be steered by changing the metric}
Let $k\geq 3$ and let $(X,g)$ be a globally hyperbolic spacetime with $g$ a $\mathcal{C}^k$ metric. Let $\gamma:[0,1]\to X$ be a non-constant geodesic, $p=\gamma(0)$ and $v=\gamma'(0)$. Let $U\subseteq X$ be a regular open set such that $\gamma^{-1}(U)\subseteq (0,1)$ is an interval, such that $\gamma|_{\gamma^{-1}(U)}:\gamma^{-1}(U)\to U$ is an embedding, and such that $\overline{U}$ is compact. Let $A=X\setminus U$, which is regular closed since $U$ is regular open. Then the map $F: \mathcal{M}_{(A,g)}^k \to X$ given by $g' \mapsto \exp^{g'}_{p}(v)$ is well-defined near $g$ and a submersion at $g$.
\end{lem}
\begin{proof}
The map $F$ is well-defined and differentiable near $g$ because the flow of the geodesic equation depends differentiably on the metric and its first derivatives. We must show its differential $dF_g: T_g(\mathcal{M}^k_{(A,g)}) \to T_{\gamma(1)}X$ is surjective. Note that $T_g(\mathcal{M}^k_{(A,g)})$ contains all $\mathcal{C}^k$ symmetric $2$-tensors compactly supported in $U$. Let $h \in T_g(\mathcal{M}^k_{(A,g)})$ and consider the variation of metrics $g_s \coloneqq g + s h$. This induces a variation of geodesics $\gamma_s(t) = \exp^{g_s}_p(tv)$. Let $J(t) \coloneqq \left.\frac{\partial}{\partial s}\right|_{s=0} \gamma_s(t)$ be the corresponding variation field along $\gamma$, with initial conditions $J(0) = 0$ and $J'(0) = 0$. Note that $J(1) = dF_g(h)$. Because $\gamma_s$ is a $g_s$-geodesic, differentiating the geodesic equation $\nabla^{g_s}_{\gamma'_s} \gamma'_s = 0$ at $s=0$ yields the inhomogeneous Jacobi equation:
$$ J'' + R(J, \gamma')\gamma' = S_h(t) $$
where $R$ is the Riemann curvature of $g$. By a standard calculation for the variation of the Levi-Civita connection, the term $S_h$ is given by the tensor:
$$ \langle S_h(t), \cdot \rangle = \frac{1}{2} (\nabla h)(\gamma', \gamma', \cdot) - (\nabla_{\gamma'} h)(\gamma', \cdot). $$
To prove $dF_g$ is surjective, we show $(\im dF_g)^\perp = \{0\}$. Let $\lambda \in (\im dF_g)^\perp$. Let $Y(t)$ be the unique Jacobi field along $\gamma$ satisfying $Y(1) = 0$ and $Y'(1) = \lambda$. Taking the inner product of the equation $ J'' + R(J, \gamma')\gamma' = S_h(t) $ with $Y(t)$ and integrating yields:
\begin{align*}
\int_0^1 \langle S_h(t), Y(t) \rangle \, dt &= \int_0^1 \langle J'' + R(J, \gamma')\gamma', Y \rangle \, dt \\
&= \int_0^1 \langle J, Y'' + R(Y, \gamma')\gamma' \rangle \, dt + \Big[ \langle J', Y \rangle - \langle J, Y' \rangle \Big]_0^1 \\
&= \langle J'(1), Y(1) \rangle - \langle J(1), Y'(1) \rangle \\
&= - \langle J(1), \lambda \rangle = 0,
\end{align*}
where in the second equality we integrated by parts twice and used the algebraic symmetry of the Riemann tensor $\langle R(J, \gamma')\gamma', Y \rangle = \langle J, R(Y, \gamma')\gamma' \rangle$, and where the integral in the second line vanishes because $Y$ is a Jacobi field. Noting that $\langle S_h,Y\rangle = \left( \frac{1}{2} (\nabla_{Y} h)(\gamma', \gamma') - (\nabla_{\gamma'} h)(\gamma', Y) \right)$ and applying integration by parts to the term $(\nabla_{\gamma'} h)(\gamma', Y)$ simplifies the integral condition to:
$$ \int_0^1 \left( \frac{1}{2} (\nabla_Y h)(\gamma', \gamma') + h(\gamma', Y') \right) dt = 0. $$
We may smoothly choose a complementary subbundle $E$ such that $T_{\gamma(t)}X = \langle \gamma'(t) \rangle \oplus E_t$. We decompose $Y(t) = c(t)\gamma'(t) + Z(t)$ with $Z(t) \in E_t$. Let $\omega$ be a smooth one-form on $X$ such that along $\gamma^{-1}(U)$ we have $\omega(\gamma') \equiv 1$ and $E_t \subseteq \ker(\omega_{\gamma(t)})$. Let $\beta(t) \coloneqq (\nabla_Y(\omega \otimes \omega))(\gamma', \gamma')$. For an arbitrary smooth function $u: X \to \R$ compactly supported in $U$, we define the tensor $h \coloneqq u \cdot (\omega \otimes \omega)$. Substituting $h$ into our integral condition and integrating the tangential derivative of $u$ by parts yields:
$$ \int_0^1 \left( \frac{1}{2} Z(u) + u(t) \underbrace{\left( \frac{1}{2}\beta(t) + \omega(Y') - \frac{1}{2}c'(t) \right)}_{\coloneqq D(t)} \right) dt = 0. $$

Notice that $D(t)$ does not depend on $u$. Let $\rho \geq 0$ be a bump function supported in $\gamma^{-1}(U)$ and let $I \subseteq \gamma^{-1}(U)$ be an open interval with $\rho|_I > 0$. Write $|Z(t)|_e^2 \coloneqq e(Z(t),Z(t))$ for the Riemannian norm induced by the auxiliary metric $e$ from \subsecref{subsec:banach}. We may take $u$ such that on $\gamma$ we have, $u(t) = \rho(t)D(t)$ and $Z(u)(t) = \rho(t)|Z(t)|_e^2$. Because $Z$ is linearly independent of $\gamma'$ wherever non-zero, this prescription is always possible. The integral becomes
$$ \int_0^1 \rho(t) \left( \frac{1}{2} |Z(t)|_e^2 + D(t)^2 \right) dt = 0. $$

Because the integrand is non-negative and $|Z(t)|_e^2$ vanishes only when $Z(t)=0$, $Z \equiv 0$ and $D \equiv 0$ on $I$. Thus $Y(t) = c(t)\gamma'(t)$ on $I$, which implies $\omega(Y') = c'$ and $\beta \equiv 0$. Substituting these into $D \equiv 0$ yields $c' \equiv 0$, meaning $Y(t) = C\gamma'(t)$ on $I$. Since $Y$ and $C\gamma'$ are both Jacobi fields which agree on $I$ we must have $Y = C\gamma'$ everywhere on $[0,1]$. Since $Y(1) = 0$ and $\gamma'(1) \neq 0$, we must have $C = 0$. Thus $Y \equiv 0$, and in particular $\lambda = Y'(1) = 0$. Therefore $\im(dF_g) = T_{\gamma(1)}X$. Because the target space is finite-dimensional, the surjective linear map $dF_g$ has a split kernel. Hence $F$ is a submersion at $g$.
\end{proof}
\begin{rem}
If one defines the space $\mathcal M^k_{(A,g)}$ for a regular closed subset $A$ of a general semi-Riemannian manifold $(X,g)$ (with any choice of auxiliary Riemannian metric $e$ and weight function $W$), \lemref{lem:Geodesics can be steered by changing the metric} still holds by exactly the same proof.
\end{rem}

\begin{lem}\label{lem:UniversalNullInitialData}
Let $(X,g)$ be a globally hyperbolic spacetime and let $A\subseteq X$ be regular closed. Let $N\geq 1$ and $k\geq 3$, and set
$$T^{(N)}X:=TX\times_X\cdots\times_X TX,$$
where there are $N$ factors. Define
\begin{align*}
\mathcal{D}_N^k
:=
\left\{
(g',p,v_1,\dots,v_N)\in \mathcal{M}_{(A,g)}^k\times T^{(N)}X
\ \middle|\
\begin{array}{l}
v_i\neq 0,\\
g'_p(v_i,v_i)=0 \text{ for all } i,\\
v_i,v_j \text{ lin. ind. for } i\neq j
\end{array}
\right\}.
\end{align*}
Then $\mathcal{D}_N^k$ is a Banach submanifold of
$\mathcal{M}_{(A,g)}^k\times T^{(N)}X$. Moreover, the projection
$$\Pi_N:\mathcal{D}_N^k\to \mathcal{M}_{(A,g)}^k$$
is a Fredholm submersion of index $d+N(d-1)$, where $d=\dim(X)$.
\end{lem}

\begin{proof}
Let
$$\Phi:\mathcal{M}_{(A,g)}^k\times T^{(N)}X\to \mathbb{R}^N$$
be defined by
\begin{align*}
\Phi(g',p,v_1,\dots,v_N)
=
\big(g'_p(v_1,v_1),\dots,g'_p(v_N,v_N)\big).
\end{align*}
The nonzero and pairwise linear-independence conditions in the definition of $\mathcal{D}_N^k$ are open conditions, so it suffices to show that $\Phi$ is a submersion along the relevant part of $\Phi^{-1}(0)$. Let $(g',p,v_1,\dots,v_N)\in \mathcal{D}_N^k$. Since $g'_p$ is nondegenerate
and $v_i\neq 0$, the covector $g'_p(v_i,\cdot)$ is nonzero. Hence, for any $(a_1,\dots,a_N)\in \mathbb{R}^N$, we can choose $\eta_i\in T_pX$ such that
$$
2g'_p(v_i,\eta_i)=a_i.
$$
Then
\begin{align*}
d\Phi_{(g',p,v_1,\dots,v_N)}(0,0,\eta_1,\dots,\eta_N)
=
(a_1,\dots,a_N).
\end{align*}
Thus $\Phi$ is a submersion along $\mathcal{D}_N^k$, and consequently
$\mathcal{D}_N^k$ is a Banach submanifold. We now show that $\Pi_N$ is a submersion. Let $h\in T_{g'}\mathcal{M}_{(A,g)}^k$. Since each covector $g'_p(v_i,\cdot)$ is nonzero, we can choose $\eta_i\in T_pX$ such that
$$
h_p(v_i,v_i)+2g'_p(v_i,\eta_i)=0.
$$
This means that $(h,0,\eta_1,\dots,\eta_N)$ is tangent to $\mathcal{D}_N^k$ at $(g',p,v_1,\dots,v_N)$. Therefore $d\Pi_N(h,0,\eta_1,\dots,\eta_N)=h$, and so $\Pi_N$ is a submersion. For fixed $g'$, the fiber $\Pi_N^{-1}(g')$ is the space of tuples $(p,v_1,\dots,v_N)$ with $v_i\in T_pX$ nonzero and $g'_p$-null, subject to the open condition that the $v_i$ are pairwise linearly independent. For each $p\in X$, the nonzero null cone in $T_pX$ has dimension $d-1$. Hence
$$
\dim\big(\Pi_N^{-1}(g')\big)=d+N(d-1).
$$
Thus $\Pi_N$ is a Fredholm submersion of index $d+N(d-1)$.
\end{proof}

\begin{prop}\label{prop: getting a residual set of an M^infinity from residual sets of M^k}
Let $k_0\in \N$. For every $k\geq k_0$, let $\mathcal{M}^k$ be a separable Banach manifold, and let $\mathcal{M}^{\infty}$ be a Fr\'echet manifold together with continuous injective maps $\iota_k:\mathcal{M}^{\infty}\hookrightarrow \mathcal{M}^k$ with dense images. Assume that the maps $\iota_k$ arise from a directed inverse system, i.e. that for all $\ell\geq k\geq k_0$ there are continuous maps
$$\rho_{k,\ell}:\mathcal{M}^{\ell}\to \mathcal{M}^k$$
such that
$$\iota_k=\rho_{k,\ell}\circ \iota_{\ell},$$
and assume that the topology on $\mathcal{M}^{\infty}$ is the corresponding inverse limit topology. For every $k\geq k_0$, let $\mathcal{P}^k$ be a separable Banach manifold and let $\pi^k:\mathcal{P}^k\to \mathcal{M}^k$ be a $\mathcal{C}^1$ Fredholm map of index at most $-1$. Assume that, for every $k\geq k_0$, there are subsets $(\mathcal{P}^k_j)_{j\in \N}$ of $\mathcal{P}^k$ such that $\mathcal{P}^k=\bigcup_{j=1}^{\infty}\mathcal{P}^k_j,$ such that $\pi^k(\mathcal{P}^k_j)$ is closed in $\mathcal{M}^k$ for all $j\in \N$ and $k\geq k_0$, and such that for all $j\in \N$ and all $k,\ell\geq k_0$ we have
$$\iota_k^{-1}\big(\pi^k(\mathcal{P}^k_j)\big) = \iota_\ell^{-1}\big(\pi^\ell(\mathcal{P}^\ell_j)\big).$$
Let $j\in\N$. We define $\mathcal{O}^k_j\coloneqq \mathcal{M}^k\setminus \pi^k(\mathcal{P}^k_j),$ for all $k\geq k_0$ and we define $\mathcal{O}^{\infty}_j
\coloneqq
\mathcal{M}^{\infty}\setminus \iota_k^{-1}\big(\pi^k(\mathcal{P}^k_j)\big),$ where $k\geq k_0$ is arbitrary. This is well-defined by the compatibility assumption.  Then each $\mathcal{O}^k_j$ is open and dense in $\mathcal{M}^k$, and $\mathcal{O}_j^{\infty}$ is open and dense in $\mathcal{M}^{\infty}$. In particular, $\mathcal{R}^k\coloneqq \bigcap_{j=1}^{\infty}\mathcal{O}^k_j$ is residual in $\mathcal{M}^k$, and $\mathcal{R}^{\infty}\coloneqq \bigcap_{j=1}^{\infty}\mathcal{O}^{\infty}_j$ is residual in $\mathcal{M}^{\infty}$. Note that by construction
$$\mathcal{R}^k\cap \pi^k(\mathcal{P}^k)=\emptyset, \quad \iota_k(\mathcal{R}^{\infty})\cap \pi^k(\mathcal{P}^k)=\emptyset$$
for every $k\geq k_0$.
\end{prop}
\begin{proof}
We first prove the statements about $k\geq k_0$. Let $k\geq k_0$. By the Sard--Smale transversality theorem \cite{Smale1965}, the set of regular values of $\pi^k$ is residual in $\mathcal M^k$. Since $\operatorname{ind}(\pi^k)\leq -1$, $\mathcal C^1$ regularity is sufficient for this application of Sard--Smale. Let $y\in \mathcal{M}^k$ be a regular value of $\pi^k$. If $(\pi^k)^{-1}(y)$ were non-empty, then by the regular level set theorem for Fredholm maps it would be a smooth manifold of dimension $\ind(\pi^k)\leq -1$, which is impossible. Hence every regular value of $\pi^k$ has empty preimage. Thus
$$\{\text{regular values of }\pi^k\}\subseteq \mathcal{M}^k\setminus \pi^k(\mathcal{P}^k).$$
In particular, for every $j\in \N$,
$$\{\text{regular values of }\pi^k\}\subseteq \mathcal{O}^k_j.$$
Since the regular values are dense, $\mathcal{O}^k_j$ is dense in $\mathcal{M}^k$. The set $\mathcal{O}^k_j$ is open because $\pi^k(\mathcal{P}^k_j)$ is closed in $\mathcal{M}^k$. Therefore $\mathcal{R}^k=\bigcap_{j=1}^{\infty}\mathcal{O}^k_j$ is residual in $\mathcal{M}^k$. If $y\in \mathcal{R}^k\cap \pi^k(\mathcal{P}^k)$, then there exists $x\in \mathcal{P}^k$ with $\pi^k(x)=y$. Since $\mathcal{P}^k=\bigcup_j\mathcal{P}^k_j$, there exists $j\in\N$ such that $x\in \mathcal{P}^k_j$. Hence $y\in \pi^k(\mathcal{P}^k_j)$, contradicting $y\in\mathcal{O}^k_j$. Thus
$$\mathcal{R}^k\cap \pi^k(\mathcal{P}^k)=\emptyset.$$
We now prove the statements regarding $\mathcal{M}^{\infty}$. The set $\mathcal{O}^{\infty}_j$ is well-defined by the compatibility assumption. To see that it is open, take any $k\geq k_0$. Clearly $\mathcal{O}^{\infty}_j=\iota_k^{-1}(\mathcal{O}^k_j)$ is open in $\mathcal{M}^{\infty}$. It remains to prove density. Let $\mathcal{V}\subseteq \mathcal{M}^{\infty}$ be a non-empty open set. Choose $x\in \mathcal{V}$. By the inverse limit topology assumption, there exist $k\geq k_0$ and an open set $\mathcal{U}\subseteq \mathcal{M}^k$ such that
$$\iota_k(x)\in \mathcal{U}
\quad\text{and}\quad
\iota_k^{-1}(\mathcal{U})\subseteq \mathcal{V}.$$
Since $\mathcal{O}^k_j$ is open and dense in $\mathcal{M}^k$, the set $\mathcal{U}\cap \mathcal{O}^k_j$ is non-empty and open in $\mathcal{M}^k$. Since $\iota_k(\mathcal{M}^{\infty})$ is dense in $\mathcal{M}^k$, there exists $z\in \mathcal{M}^{\infty}$ such that $\iota_k(z)\in \mathcal{U}\cap \mathcal{O}^k_j.$ Then $z\in \iota_k^{-1}(\mathcal{U})\subseteq \mathcal{V}$, and also $z\in \iota_k^{-1}(\mathcal{O}^k_j)=\mathcal{O}^{\infty}_j.$ Thus $\mathcal{V}\cap \mathcal{O}^{\infty}_j\neq \emptyset.$ Since $\mathcal{V}$ was arbitrary, $\mathcal{O}^{\infty}_j$ is dense in $\mathcal{M}^{\infty}$. Therefore $\mathcal{R}^{\infty}=\bigcap_{j=1}^{\infty}\mathcal{O}^{\infty}_j$ is residual in $\mathcal{M}^{\infty}$. Finally, suppose that $z\in \mathcal{R}^{\infty}$ and that for some $k\geq k_0$ we have $\iota_k(z)\in \pi^k(\mathcal{P}^k).$ Then $\iota_k(z)=\pi^k(x)$ for some $x\in \mathcal{P}^k$. Choose $j\in\N$ with $x\in \mathcal{P}^k_j$. Then $z\in \iota_k^{-1}\big(\pi^k(\mathcal{P}^k_j)\big),$ so $z\notin \mathcal{O}^{\infty}_j$, contradicting $z\in \mathcal{R}^{\infty}$. Hence
$$\iota_k(\mathcal{R}^{\infty})\cap \pi^k(\mathcal{P}^k)=\emptyset$$
for every $k\geq k_0$.
\end{proof}

\begin{thm}\label{thm: For N large enough generic spacetimes are not N-refocusing}
Let $(X,g)$ be a globally hyperbolic spacetime and let $A\subseteq X$ be a regular closed set. Let $d=\dim(X)$ and let $N\geq 2d+1$. For $k\in\N_{\geq 3}\cup\{\infty\}$ there exists a residual set
$$\mathcal{R}^k \subseteq \mathcal{M}_{(A,g)}^k$$
such that no metric $g'\in \mathcal{R}^k$ is $N$-fold null-geodesically connecting through $X\setminus A$. 
\end{thm}

\begin{proof}
Let $k\geq 3$. Let $\mathcal{D}_N^k$ be as in \lemref{lem:UniversalNullInitialData}. Let $\mathcal{Q}^k\subseteq \mathcal{D}_N^k$ be the open set of points $(g',p,v_1,\dots,v_N)$ for which every null-geodesic segment $t\mapsto \exp^{g'}_p(tv_i)$, $t\in [0,1]$, is well-defined and intersects $X\setminus A$. We define the evaluation map 
$$E_k:\mathcal{Q}^k \to X^N,\quad E_k(g',p,v_1,\dots,v_N)=(\exp_p^{g'}(v_1),\dots,\exp_p^{g'}(v_N)).$$
Let $\Delta\subseteq X^N$ denote the diagonal. Since $\dim(\Delta)=\dim(X)=d$, we have $\codim(\Delta)=\dim(X^N) - \dim(\Delta) = dN - d = d(N-1)$. Let $\mathcal{P}^k\coloneqq E_k^{-1}(\Delta)$. Let $(g', \mathbf{x}) \in \mathcal{P}^k$, where $\mathbf{x}=(p,v_1,\dots,v_N)$, and write
$$q:=\exp_p^{g'}(v_1)=\cdots=\exp_p^{g'}(v_N).$$
Because the $N$ initial null directions $v_i$ are pairwise linearly independent, the resulting null geodesics $\gamma_i(t)=\exp_p^{g'}(tv_i)$ are distinct. Thus, their intersection sets are discrete, and we can choose $N$ mutually disjoint regular open sets $U_1, \dots, U_N \subseteq X \setminus A$, with compact closures contained in $X\setminus A$, such that $U_i$ intersects $\gamma_i$ away from its endpoints and avoids all other rays. Since $(X,g')$ is globally hyperbolic hence strongly causal, we can pick each $U_i$ small enough to ensure that $\gamma_i^{-1}(U_i)$ is an interval. Because the perturbations in each $U_i$ can be chosen independently, and because each $U_i$ is disjoint from $p$, \lemref{lem:Geodesics can be steered by changing the metric} applied to $g'$ as the base metric implies that the differential of $E_k$ is surjective onto $T_{(q,\dots,q)} X^N$. The condition that each $U_i$ is disjoint from $p$ ensures that these perturbations do not change the equations $g'_p(v_i,v_i)=0$, hence they define tangent vectors to $\mathcal{D}_N^k$. Therefore, $E_k$ is a submersion at $(g', \mathbf{x})$, so $E_k \pitchfork \Delta$. Hence $\mathcal{P}^k$ is a Banach submanifold of $\mathcal{Q}^k$. Further $\codim(\mathcal{P}^k) = \codim(\Delta) =d(N-1)$. Let $\Pi^k: \mathcal{D}_N^k \to \mathcal{M}_{(A,g)}^k$ be the projection onto the first factor. By \lemref{lem:UniversalNullInitialData}, $\Pi^k$ is a Fredholm map with $\ind(\Pi^k)=d+N(d-1)$. Let $\jmath_k: \mathcal{P}^k \hookrightarrow \mathcal{Q}^k$ denote the inclusion map. The differential $d\jmath_k$ is everywhere injective, so $\dim(\ker d\jmath_k) = 0$. The cokernel of $d\jmath_k$ has dimension $\codim(\mathcal{P}^k)$. Thus, $\jmath_k$ is a Fredholm map with $\ind(\jmath_k) = -\codim(\mathcal{P}^k)$. The restricted projection map is given by the composition $\pi^k = \Pi^k \circ \jmath_k: \mathcal{P}^k \to \mathcal{M}_{(A,g)}^k$. By the additivity of the Fredholm index under composition, $\pi^k$ is Fredholm, and its index is:
\begin{align*}
\ind(\pi^k) &= \ind(\Pi^k) + \ind(\jmath_k)\\
&= [d+ N(d-1)] - \codim(\mathcal{P}^k)\\
&= [d+ N(d-1)] - d(N-1)\\
&= d + dN - N - dN + d\\
&= 2d - N\\
&\leq -1.
\end{align*}
Next we show that we can find subsets $(\mathcal{P}_j^k)_{j\in \N}$ of $\mathcal{P}^k$ such that $\mathcal{P}^k = \bigcup_{j\in \N} \mathcal{P}^k_j$, such that each $\pi^k(\mathcal{P}_j^k)\subseteq \mathcal{M}^k_{(A,g)}$ is closed and such that $\pi^{k_2}(\mathcal{P}^{k_2}_j) = \pi^{k_1}(\mathcal{P}^{k_1}_j) \cap \mathcal{M}^{k_2}_{(A,g)}$ for all $k_1,k_2 \in \N_{\geq 3}$ with $k_1<k_2$. Let $\mathcal{Y}\subseteq T^{(N)}X$ denote the open subset consisting of tuples $(p,v_1,\dots,v_N)$ with all $v_i$ non-zero and pairwise linearly independent. Choose exhaustions by compact sets
$$B_1\subseteq B_2\subseteq \cdots \subseteq \mathcal{Y},
\qquad
H_1\subseteq H_2\subseteq \cdots \subseteq TX,
\qquad
L_1\subseteq L_2\subseteq \cdots \subseteq X\setminus A$$
with $\bigcup_j B_j=\mathcal{Y}$, $\bigcup_j H_j=TX$, and $\bigcup_j L_j=X\setminus A$ and with the property that every compact subset of the relevant space is contained in some member of the exhaustion. For $j\in\N$, let $\mathcal{P}_j^k\subseteq \mathcal{P}^k$ be the subset consisting of points $(g',p,v_1,\dots,v_N)$ such that $(p,v_1,\dots,v_N)\in B_j$, such that the lifted geodesic segments
$$
t\mapsto \left(\exp_p^{g'}(tv_i),\frac{d}{dt}\exp_p^{g'}(tv_i)\right)
$$
are contained in $H_j$ for all $i=1,\dots,N$, and such that each geodesic segment $t\mapsto \exp_p^{g'}(tv_i)$ intersects $L_j$. The sets $\mathcal{P}_j^k$ exhaust $\mathcal{P}^k$. Indeed, for any point $(g',p,v_1,\dots,v_N)\in \mathcal{P}^k$, the initial data $(p,v_1,\dots,v_N)$ lie in some $B_j$, the finitely many lifted geodesic segments are compact subsets of $TX$ and hence lie in some $H_j$, and, since each segment intersects $X\setminus A$, the finitely many intersections are detected in some $L_j$. We claim that $\pi^k(\mathcal{P}_j^k)$ is closed in $\mathcal{M}^k_{(A,g)}$. Let $(g'_m)_{m\in\N}$ be a sequence in $\pi^k(\mathcal{P}_j^k)$ converging to $g'$ in $\mathcal{M}^k_{(A,g)}$. Choose corresponding initial data $(p_m,v_{1,m},\dots,v_{N,m})\in B_j$. Since $B_j$ is compact, after passing to a subsequence we may assume that
$$(p_m,v_{1,m},\dots,v_{N,m})\to (p,v_1,\dots,v_N)\in B_j.$$
By continuous dependence of geodesics on the metric and initial data, and since the lifted geodesic segments are contained in the compact set $H_j$, the limiting $g'$-geodesic segments
$$
\gamma_i(t)=\exp_p^{g'}(tv_i)
$$
are defined on $[0,1]$, their lifted images are contained in $H_j$, and
$$\gamma_1(1)=\cdots=\gamma_N(1).$$
Moreover, since each approximating segment intersects the compact set $L_j$, after passing to subsequences in the parameters $t\in[0,1]$ each limiting segment also intersects $L_j$. Finally, the null conditions $g'_p(v_i,v_i)=0$ follow by taking limits of the equations $g'_{m,p_m}(v_{i,m},v_{i,m})=0$. Thus $(g',p,v_1,\dots,v_N)\in \mathcal{P}_j^k$, so $g'\in \pi^k(\mathcal{P}_j^k)$. Hence $\pi^k(\mathcal{P}_j^k)$ is closed. It is easy to see that  $\pi^{k_2}(\mathcal{P}^{k_2}_j) = \pi^{k_1}(\mathcal{P}^{k_1}_j) \cap \mathcal{M}^{k_2}_{(A,g)}$ for all $k_1,k_2 \in \N_{\geq 3}$ with $k_1<k_2$. Finally, note that $\iota_k(\mathcal{M}^{\infty}_{(A,g)})\subseteq \mathcal{M}^k_{(A,g)}$ is dense for all $k\in \N_{\geq 3}$. We now apply \propref{prop: getting a residual set of an M^infinity from residual sets of M^k} to the sequence of Banach manifolds $(\mathcal{M}^k_{(A,g)})_{k\geq 3}$ and the Fr\'echet manifold $\mathcal{M}^{\infty}_{(A,g)}$ with the natural inclusions $\rho_{k_1,k_2}:\mathcal{M}^{k_2}_{(A,g)}\to \mathcal{M}^{k_1}_{(A,g)}$ and $\iota_k:\mathcal{M}^{\infty}_{(A,g)}\to \mathcal{M}^k_{(A,g)}$. As a result we obtain residual sets $\mathcal{R}^k\subseteq \mathcal{M}^k_{(A,g)}$ with $\mathcal{R}^k \cap \pi^k(\mathcal{P}^k) = \emptyset$ and a residual set $\mathcal R^\infty\subseteq \mathcal M^\infty_{(A,g)}$ such that
$$\iota_k(\mathcal R^\infty)\cap \pi^k(\mathcal P^k)=\emptyset$$
for every $k\geq 3$. In particular, no metric in any of the residual sets $\mathcal{R}^k$, $k\in\N_{\geq 3}\cup\{\infty\}$, is $N$-fold null-geodesically connecting through $X\setminus A$.
\end{proof}

\begin{cor}\label{cor: Avoiding refocusing in A means we can avoid refocusing in X}
Let $(X,g)$ be a globally hyperbolic spacetime of dimension at least $3$ and let $A\subseteq X$ be regular closed such that $g$ is not strongly refocusing within $A$. Then there exists a residual set $\mathcal{R}\subseteq \mathcal{M}^{\infty}_{(A,g)}$ such that $(X,g')$ is not strongly refocusing for any $g'\in \mathcal{R}$.
\end{cor}

\begin{proof}
By \theoref{thm: For N large enough generic spacetimes are not N-refocusing}, there exists a residual set $\mathcal{R}\subseteq \mathcal{M}^{\infty}_{(A,g)}$ of metrics which are not $N$-fold null-geodesically connecting through $X\setminus A$. Take $g'\in \mathcal{R}$ and $p,q\in X$, $p\neq q$. Since $(X,g)$ is not strongly refocusing within $A$, neither is $(X,g')$ by \remref{rem: strongly refocusing within a subset is stable}. If $g'$ were strongly refocusing with respect to $p$ and $q$, then we must have $C_{g'}(p,q) \cap (X\setminus A)\neq \emptyset$, as else $(X,g')$ would be strongly refocusing within $A$. Let $C_r\subseteq C_{g'}(p,q)$ denote the regular part of $C_{g'}(p,q)$. Then $C_r$ is a $\dim(X)-1$ dimensional null hypersurface in $(X,g')$. Since $X\setminus A$ is open and since $C_r\subseteq C_{g'}(p,q)$ is dense, $C_r\cap (X\setminus A)$ is non-empty and open in $C_r$, hence contains a smooth submanifold of dimension $\dim(X)-1\geq 2$. If there were only finitely many $g'$-null-geodesic segments connecting $p$ and $q$ through $X\setminus A$, then $C_r\cap (X\setminus A)$ would be contained in a finite union of $1$-dimensional images, which is impossible. Thus there are infinitely many, and in particular $N=2\dim(X)+1$, $g'$-null-geodesics connecting $p$ to $q$ through $X\setminus A$, contradicting the fact that $g'\in\mathcal R$. So $g'$ cannot be strongly refocusing with respect to any two distinct points $p,q\in X$.
\end{proof}

\subsection{Proof of the main theorem}

\begin{defin}\label{def:preorder on pairs of points w.r.t. which a spacetime is strongly refocusing}
Let $(X,g)$ be a strongly refocusing globally hyperbolic spacetime and let $\mathcal{T}:X\to \R$ be a Cauchy temporal function. Let
$$R\coloneqq\{(p,q)\in X\times X \mid q\in J^+(p), (X,g) \text{ strongly refocusing with respect to } p,q\}.$$
We define a preorder $\preceq_{\mathcal{T}}$ on $R$ by saying that
$$(p_1,q_1)\preceq_{\mathcal{T}} (p_2,q_2) \iff [\mathcal{T}(p_1),\mathcal{T}(q_1)] \subseteq [\mathcal{T}(p_2),\mathcal{T}(q_2)].$$
We write $(p_1,q_1)\prec_{\mathcal T}(p_2,q_2)$ if $(p_1,q_1)\preceq_{\mathcal T}(p_2,q_2)$ but $(p_2,q_2)\not\preceq_{\mathcal T}(p_1,q_1)$. We say an element $(p,q)\in (R,\preceq_{\mathcal{T}})$ is \textit{minimal} if there exists no $(p',q')\in R$ with $(p',q')\prec_{\mathcal{T}} (p,q) $.
\end{defin}

\begin{lem}\label{lem:minimal elements exist}
Let $(X,g)$ be a strongly refocusing globally hyperbolic spacetime and $\mathcal{T}:X\to \R$ a Cauchy temporal function. Let $(R,\preceq_{\mathcal{T}})$ be as in \defref{def:preorder on pairs of points w.r.t. which a spacetime is strongly refocusing}. Then $(R,\preceq_{\mathcal{T}})$ has a minimal element.
\end{lem}
\begin{proof}
It is easy to see that $R\subseteq X\times X$ is closed. Choose any $(p_0,q_0)\in R$ and set
$$D\coloneqq\{(p,q)\in R \mid (p,q)\preceq_{\mathcal{T}}(p_0,q_0)\}.$$
Since $(X,g)$ is strongly refocusing, hence refocusing, any Cauchy surface of $(X,g)$ is compact (see \cite{Low2006}). Hence
$$\mathcal{T}^{-1}([\mathcal{T}(p_0),\mathcal{T}(q_0)])\simeq M\times[\mathcal{T}(p_0),\mathcal{T}(q_0)]$$
is compact, where $M$ is a Cauchy surface of $X$. Therefore $D$ is a compact subset of $R\subseteq X\times X$. Define
$$L:D\to\R,\quad L(p,q)\coloneqq \mathcal{T}(q)-\mathcal{T}(p).$$
By compactness, $L$ attains a minimum at some $(p',q')\in D$. It is easy to see that $(p',q')$ is minimal in $(R,\preceq_{\mathcal{T}})$.
\end{proof}

\begin{lem}\label{lem:cutting makes strong refocusing refocusing}
Let $(X,g)$ be strongly refocusing with respect to points $p,q\in X$, $q\in J^+(p)$. Let $\mathcal{T}:X\to \R$ be a surjective Cauchy temporal function, $s=\mathcal{T}(p)$ and $t=\mathcal{T}(q)$. Let $X'=\mathcal{T}^{-1}((-\infty,t))$. Then $(X',g|_{X'})$ is Legendrian refocusing within $A = \mathcal{T}^{-1}([s,t))$.
\end{lem}
\begin{proof}
Note that the spaces of lightrays $\mathcal{N}_X$ of $(X,g)$ and $\mathcal{N}_{X'}$ of $(X',g|_{X'})$ are canonically diffeomorphic through the map $\mathcal{N}_X \to \mathcal{N}_{X'}$ given by restricting an unparametrized lightray in $X$ to its segment in $X'$. This yields a canonical identification of the space of Legendrian submanifolds $\mathcal{L}_X$ of $\mathcal{N}_X$ with the space of Legendrian submanifolds $\mathcal{L}_{X'}$ of $\mathcal{N}_{X'}$ Legendrian isotopic to a sky. Take any sequence $(q_n)_{n\in \N}$ in $A$ converging towards $q$ in $X$. Then, in $\mathcal{L}_X$ we have
$$\lim_{n\to \infty} \sky^X_{q_n} = \sky^X_q = \sky^X_p,$$
where $\sky^X_p=\sky^X_q$ holds because $(X,g)$ is strongly refocusing with respect to the pair $(p,q)$. Hence in $\mathcal{L}_{X'}$ we get
$$\lim_{n\to \infty} \sky^{X'}_{q_n} = \sky^{X'}_p,$$
even though $q_n \nrightarrow p$. Let $M=\mathcal{T}^{-1}(s)\subseteq A$. Then every null-geodesic segment connecting a point $q_n$ to $M$ lies entirely in $A$. We conclude that $(X',g|_{X'})$ is Legendrian refocusing within $A$ with respect to $p$ and the sequence $(q_n)_{n\in \N}$.
\end{proof}

\begin{thm}\label{thm:main}
Every globally hyperbolic strongly refocusing spacetime $(X,g)$ of dimension at least $3$ admits a globally hyperbolic metric $g'$ on $X$ which is Legendrian refocusing (hence in particular refocusing) but not strongly refocusing.
\end{thm}
\begin{proof}
Let $\mathcal{T}:X\to \R$ be a surjective Cauchy temporal function. Let $(R,\preceq_{\mathcal{T}})$ be as in \defref{def:preorder on pairs of points w.r.t. which a spacetime is strongly refocusing} and let $(p,q)\in (R,\preceq_{\mathcal{T}})$ be a minimal element. Let $s = \mathcal{T}(p)$ and let $t=\mathcal{T}(q)$. Let $X' = \mathcal{T}^{-1}((-\infty,t))$. Then $(X',g|_{X'})$ is globally hyperbolic. By \lemref{lem:cutting makes strong refocusing refocusing}, $(X',g|_{X'})$ is Legendrian refocusing within $A\coloneqq \mathcal{T}^{-1}([s,t))$. If $(X',g|_{X'})$ were strongly refocusing with respect to a pair of points $p',q'\in A$ with $q'\in J^+(p')$, then we would have $(p',q')\prec_{\mathcal{T}}(p,q)$, contradicting the minimality of $(p,q)$. We conclude $(X',g|_{X'})$ is not strongly refocusing within $A$. Since $A\subseteq X'$ is a regular closed subset, by \corolref{cor: Avoiding refocusing in A means we can avoid refocusing in X}, there exists a globally hyperbolic metric $g'\in \mathcal{M}_{(A,g|_{X'})}^{\infty}$ so that $(X',g')$ is not strongly refocusing. Note that by \propref{prop: Legendrian refocusing within a subset is stable}, $(X',g')$ is Legendrian refocusing within $A$. Since $X$ and $X'$ are diffeomorphic we can simply choose a diffeomorphism $\Phi:X\to X'$ and consider the spacetime $(X,\Phi^{\ast}g')$, which is globally hyperbolic, Legendrian refocusing, but not strongly refocusing.
\end{proof}

\section{Legendrian refocusing spacetimes admit strongly refocusing metrics}\label{sec:Legendrian refocusing spacetimes admit strongly refocusing metrics}

\subsection{Scattering of null-geodesic fields}
In the two-part paper \cite{GluckSinger1978, GluckSinger1979}, Gluck and Singer study the scattering of geodesic fields on Riemannian manifolds. In particular they develop a necessary and sufficient criterion for being able to deflect a geodesic vectorfield onto another \cite[Theorem~1]{GluckSinger1978}. We mirror their techniques to prove the following lemma, which is similar in spirit to their work, giving a sufficient criterion for being able to deflect one null-geodesic vectorfield onto another.

\begin{lem}\label{lem:scattering lemma}
Let $M$ be a smooth, connected manifold and let $S\subset M$ be a compact hypersurface (without boundary) with collar neighborhood $C=S\times[-\varepsilon,\varepsilon]\subseteq M$. Let $\delta>0$. Suppose that a smooth family of Riemannian metrics $h_t$ on $M$ is given for $t\in[-\delta,0]\cup[1,1+\delta]$. Let $V_0\in \Gamma(TM|_{S\times\{-\varepsilon\}})$ be the $h_0$-unit normal to $S\times\{-\varepsilon\}$ pointing into $C$, and let $V_1\in \Gamma(TM|_{S\times\{\varepsilon\}})$ be the $h_1$-unit normal to $S\times\{\varepsilon\}$ pointing out of $C$. Then there exists a diffeomorphism $\Phi:S\times [0,1] \to C$ and an extension of the family of metrics $\{h_t\}$ to all $t\in[-\delta,1+\delta]$ such that
\begin{itemize}
\item $\Phi(p,0)=(p,-\varepsilon)$ and
$\Phi(p,1)=(p,\varepsilon)$;
\item for every $p\in S$, the curve
$\gamma_p(t)=(\Phi(p,t),t)$ is a null pregeodesic for $g\coloneqq h_t-dt^2$;
\item $\frac{d}{dt}|_{t=0}\Phi(p,t)=V_0(p)$ and
$\frac{d}{dt}|_{t=1}\Phi(p,t)=V_1(p)$.
\end{itemize}
\end{lem}

\begin{proof}
Let $\{\hat{h}_{t}\}_{t\in[-\delta,1+\delta]}$ be a smooth extension of the given family, and put $\hat{g}=\hat{h}_{t}-dt^{2}$. We first construct the diffeomorphism $\Phi$. Write $S_{-}:=S\times\{-\varepsilon\}$ and $S_{+}:=S\times\{\varepsilon\}$. The vectorfield $V_{0}+\partial_{t}$ along $S_{-}\times\{0\}$ is $\hat{g}$-null and $\hat{g}$-orthogonal to $S_{-}\times\{0\}$, and similarly $V_{1}+\partial_{t}$ along $S_{+}\times\{1\}$ is $\hat{g}$-null and $\hat{g}$-orthogonal to $S_{+}\times\{1\}$. Using the corresponding null-geodesic flows near the two boundary components, parametrized by the time-coordinate, and then smoothly joining the resulting short collars in the middle, we obtain a diffeomorphism
$$\Phi:S\times[0,1]\to C$$
with the following properties:
\begin{itemize}
\item $\Phi(p,0)=(p,-\varepsilon)$ and $\Phi(p,1)=(p,\varepsilon)$;
\item there exists $c\in(0,\frac{1}{2})$ such that, for every $p\in S$, the curve $\gamma_{p}(t)=(\Phi(p,t),t)$ is a null pregeodesic for $\hat{g}$ on $[0,c]\cup[1-c,1]$;
\item $\partial_{t}\Phi(p,0)=V_{0}(p)$ and $\partial_{t}\Phi(p,1)=V_{1}(p)$.
\end{itemize}
In particular, we have on $S\times([0,c]\cup[1-c,1])$, that
$$\hat{h}_{t}(\partial_{t}\Phi,\partial_{t}\Phi)=1,\quad \hat{h}_{t}(\partial_{t}\Phi,d_{p}\Phi_{t}(T_{p}S))=0,$$
where the second equation is due to the Gauss lemma. We now fit a metric to the curves determined by $\Phi$. At $x=\Phi(p,t)$, set
$$U_{x}:=\partial_{t}\Phi(p,t),\quad E_{x}:=d_{p}\Phi_{t}(T_{p}S)\subset T_{x}M.$$
Since $\Phi$ is a diffeomorphism, $TC=\R U\oplus E$. Let $m$ be any Riemannian metric on $M$ such that, on $C$,
$$m(U,U)=1,\quad m(U,v)=0 \text{ for all } v\in E.$$
Choose a smooth function $f:[-\delta,1+\delta]\to[0,1]$ such that
$$f=0 \text{ on }[-\delta,0]\cup[1,1+\delta],\quad f=1 \text{ on }[c,1-c].$$
Define
$$h_{t}:=f(t)m+(1-f(t))\hat{h}_{t}.$$
This is a smooth family of Riemannian metrics, and since $f=0$ on $[-\delta,0]\cup[1,1+\delta]$, it extends the originally prescribed family. Along the graph of $\Phi$ we have
$$h_{t}(U,U)=1,\quad h_{t}(U,v)=0 \text{ for all } v\in E.$$
Indeed, on $S\times([0,c]\cup[1-c,1])$ this holds for both $\hat{h}_{t}$ and $m$, hence for their convex combinations, and on $S\times[c,1-c]$ it holds because $h_{t}=m$. Let
$$F:S\times[0,1]\to M\times[0,1],\quad F(p,t)=(\Phi(p,t),t),$$
let $\mathcal{H}=F(S\times[0,1])$, and put $K:=F_{\ast}(\partial_{t})$. The integral curves of $K$ are exactly $\gamma_{p}(t)=(\Phi(p,t),t)$. At $F(p,t)$ one has $K=(U,\partial_{t})$, while for $X\in T_{p}S$ we have $F_{\ast}X=(d_{p}\Phi_{t}(X),0)$. Therefore, for $g=h_{t}-dt^{2}$,
$$g(K,K)=h_{t}(U,U)-1=0,\quad g(K,F_{\ast}X)=h_{t}(U,d_{p}\Phi_{t}(X))=0,$$
where the second equality holds because $d_{p}\Phi_{t}(X)\in E$. Since $T\mathcal{H}$ is spanned by $K$ and the vectors $F_{\ast}X$, the field $K$ is both tangent and normal to $\mathcal{H}$. Hence $\mathcal{H}$ is a null hypersurface and $K$ is a normal field on $\mathcal{H}$. A null normal field on a null hypersurface is pregeodesic, thus the integral curves $\gamma_{p}(t)=(\Phi(p,t),t)$ of $K$ are null pregeodesics with respect to $g=h_{t}-dt^{2}$. 
\end{proof}

\subsection{Constructing a strongly refocusing spacetime from a Legendrian refocusing one}

\begin{lem}
Let $M$ be a smooth manifold and let $S\subset M$ be a compact hypersurface in $M$ separating $M$ into two components. Let $U\subseteq M\setminus S$ be one of the two components. Then there exists a diffeomorphism $\varphi:M\to M$ isotopic to the identity such that $\varphi(S)\subset U$ and such that $S\sqcup \varphi(S)$ bounds a set in $M$ diffeomorphic to $S\times[0,1]$.
\end{lem}
\begin{proof}
Choose a vectorfield along $S$ pointing into $U$, extend it to a compactly supported vectorfield $W$ on $M$, and let $\varphi$ be the time-$\varepsilon$ flow of $W$ for $\varepsilon>0$ sufficiently small. Then $\varphi$ is isotopic to the identity, $\varphi(S)\subset U$, and the flow $S\times[0,\varepsilon]\to M$ is a diffeomorphism onto its image, which is bounded by $S$ and $\varphi(S)$.
\end{proof}

\begin{lem}
Let $(X,g)$ be a globally hyperbolic, Legendrian refocusing spacetime. Then there exist a spacelike Cauchy surface $M\subseteq X$, a point $p\in I^{-}(M)$ and a point $q\in I^{+}(M)$ such that the fronts
$$S_{p}(M)\coloneqq \{x\in M \mid \exists \text{ null geodesic } \alpha \text{ with } \alpha(0)=p,\alpha(1)=x\}$$
and
$$S_{q}(M)\coloneqq \{x\in M \mid \exists \text{ null geodesic } \alpha \text{ with } \alpha(0)=q,\alpha(1)=x\}$$
are embedded compact hypersurfaces of $M$, with $S_{p}(M)$ separating $M$, and there exists a diffeomorphism $\varphi:(M,S_{q}(M))\to (M,S_{p}(M))$ of pairs, isotopic to the identity, which carries the coorientation of $S_q(M)$ determined by $\sky_q$ to the coorientation of $S_p(M)$ determined by $\sky_p$.
\end{lem}
\begin{proof}
Let $p\in X$ and let $(q_n)_{n\in\N}$ be a sequence with respect to which $(X,g)$ is Legendrian refocusing. After passing to a subsequence and reversing the time-orientation we can choose a spacelike Cauchy surface $M$ with $p\in I^-(M)$ and with $q_n\in I^+(M)$ for all sufficiently large $n$.  We can pick $M$ sufficiently close to $p$ to ensure that $S_p(M)$ is an embedded sphere bounding a ball of $M$, hence separates $M$. Let $\psi_M:\mathcal{N}\to ST^{\ast}M$ be the natural contactomorphism and let $\pi:ST^{\ast}M\to M$ be the projection. Since $\pi|_{\psi_M(\sky_p)}$ is an embedding and since $\sky_{q_n}\to \sky_p$ in $\mathcal{L}$, the restriction $\pi|_{\psi_M(\sky_{q_n})}$ is also an embedding for all sufficiently large $n$. Thus $S_{q_n}(M)=\pi(\psi_M(\sky_{q_n}))$ is an embedded compact hypersurface, $C^\infty$-close to $S_p(M)$ as a cooriented hypersurface. The normal projection from a small tubular neighborhood of $S_p(M)$ gives a diffeomorphism $S_{q_n}(M)\to S_p(M)$ preserving the induced coorientations. By the isotopy extension theorem this extends to a diffeomorphism of pairs $\varphi:(M,S_{q_n}(M))\to(M,S_p(M))$ isotopic to the identity. Taking $q=q_n$ for such an $n$ proves the claim.
\end{proof}

\begin{thm}\label{thm:Legendrian refocusing implies existence of strongly refocusing metric}
Let $(X,g)$ be a globally hyperbolic spacetime which is Legendrian refocusing. Then $X$ admits a globally hyperbolic metric $g'$ which is strongly refocusing.
\end{thm}

\begin{proof}
By the previous lemma, choose $p,q,M$, the fronts $S_p=S_p(M)$ and $S_q=S_q(M)$, and a diffeomorphism $\varphi:(M,S_q)\to(M,S_p)$ isotopic to the identity and preserving the coorientations induced by the two skies. Since $(X,g)$ is Legendrian refocusing, hence in particular refocusing, $M$ is compact. Choose a smooth surjective Cauchy temporal function $T : X \to \mathbb{R}$
with $T^{-1}(0)=M$. By the smooth splitting theorem of Bernal and S\'anchez \cite{BernalSanchez2005}, the corresponding splitting identifies $X$ with $M\times\mathbb{R}$. Under this identification, the metric takes the form
$$g=h_t- A dt^2 $$
for a smooth positive function $A$. Since conformal changes preserve unparametrized null geodesics and hence preserve Legendrian refocusing, after
rescaling we may assume that
$$ g=h_t-dt^2.$$
Let $V_p$ be the $h_0$-unit normal to $S_p$ determined by the future-directed null geodesics from $p$ to $M$, so that the corresponding null vectorfield along $S_p$ is $V_p+\partial_t$. Similarly, let $V_q$ be the $h_0$-unit normal to $S_q$ determined by the future-directed null geodesics from $M$ to $q$, so that the corresponding null vectorfield along $S_q$ is $V_q+\partial_t$. Let $U$ be the component of $M\setminus S_p$ into which $V_p$ points. By the first lemma there is a diffeomorphism $\psi:M\to M$, isotopic to the identity, such that $\psi(S_p)\subset U$ and such that $S_p\sqcup\psi(S_p)$ bounds a collar $C\subseteq M$. Set
$$\theta:=\psi\circ\varphi.$$
Then $\theta(S_q)=\psi(S_p)$, and the vectorfield $\theta_{\ast}V_q$ is the unit normal to $\theta(S_q)$ pointing out of the collar, with respect to the metric $\theta_{\ast}h_0$. Choose $\delta>0$. On the time intervals $[-\delta,0]$ and $[1,1+\delta]$, prescribe a family of Riemannian metrics on $M$ by
$$\bar{h}_t=h_t \text{ for } t\in[-\delta,0],\quad \bar{h}_t=\theta_{\ast}h_{t-1} \text{ for } t\in[1,1+\delta].$$
We apply \lemref{lem:scattering lemma} to the compact separating hypersurface $S_p\subset M$ with collar $C$, lower normal $V_p$, upper normal $\theta_{\ast}V_q$, and the family of Riemannian metrics $\{\bar{h}_t\}_{t\in [-\delta,0]\cup[1,1+\delta]}$. We obtain an extension of the family $\{\bar{h}_t\}$ to $t\in[-\delta,1+\delta]$ and a diffeomorphism
$$\Phi:S_p\times[0,1]\to C$$
such that $\Phi(x,0)=x$, $\Phi(x,1)\in\theta(S_q)$, and the curves
$$\Gamma_x(t)=(\Phi(x,t),t)$$
are null pregeodesics for $\bar{h}_t-dt^2$. We define a smooth family of Riemannian metrics $h'_t$ on $M$ by
$$ h'_t=
\begin{cases}
h_t, & t\leq 0,\\
\bar{h}_t, & 0\leq t\leq 1,\\
\theta_{\ast}h_{t-1}, & t\geq 1.
\end{cases} $$
and set
$$g':=h'_t-dt^2.$$
Consider the projection $\mathcal{T}:M\times \R \to \R$. Then
$\grad_{g'}\mathcal{T}=-\partial_t$ is timelike and past-pointing with respect
to $g'$, hence $\mathcal{T}$ is a temporal function for $g'$. Since $M$ is
compact, $\mathcal{T}$ is proper. Therefore $\mathcal{T}$ is a Cauchy temporal function by \cite[Section~3.2]{BurtscherGarciaHeveling2024}, and hence
$(M\times\R,g')$ is globally hyperbolic. It remains to check that $(M\times \R,g')$ is strongly refocusing. Write $q=(y,\tau)$ in the splitting $X=M\times\R$, with $\tau>0$, and set $q'=(\theta(y),\tau+1)$. The map
$$\Theta:M\times[0,\infty)\to M\times[1,\infty),\quad \Theta(x,t)=(\theta(x),t+1)$$
is an isometry from $(M\times[0,\infty),g)$ to $(M\times[1,\infty),g')$. Hence the future-directed null geodesics from $\theta(S_q)\times\{1\}$ with initial vectorfield $\theta_{\ast}V_q+\partial_t$ all pass through $q'$. Now let $\alpha$ be any future-directed $g'$-null geodesic through $p$. Since $g'=g$ on $M\times(-\infty,0]$, the curve $\alpha$ reaches $M\times\{0\}$ at a point $x\in S_p$ with tangent $V_p+\partial_t$, up to positive reparametrization. By our construction, its continuation through the slab $(M\times[0,1],g')$ is the curve $\Gamma_x$, and it reaches $\theta(S_q)\times\{1\}$ with tangent $\theta_{\ast}V_q+\partial_t$, again up to positive reparametrization. In $(M\times [1,\infty),g')$, we then get that $\alpha$ passes through $q'$. Thus every null geodesic through $p$ passes through $q'$, so $(X,g')$ is strongly refocusing.
\end{proof}

\begin{cor}\label{cor:Legendrian refocusing implies integral cohomology of a CROSS}
Let $(X,g)$ be a globally hyperbolic spacetime of dimension at least $3$ which is Legendrian refocusing. Then any Cauchy surface of $(X,g)$ is compact with finite fundamental group, and its universal cover has the integral cohomology ring of a compact rank one symmetric space (CROSS).
\end{cor}

\begin{proof}
By \theoref{thm:Legendrian refocusing implies existence of strongly refocusing metric}, the smooth manifold $X$ admits a globally hyperbolic strongly refocusing metric $g'$. In the construction above, $g'$ is defined on the same product $X=M\times\R$, so $M$ is also a Cauchy surface for $(X,g')$. The claim follows from \propref{prop:BottSamelsonForStronglyRefocusingSpacetime}.
\end{proof}

We have shown that Legendrian refocusing spacetimes admit strongly refocusing metrics. The following two questions are open. 
\begin{question}\label{quest: Is every refocusing spacetime Legendrian refocusing}
Is every globally hyperbolic spacetime which is refocusing (at a point) also Legendrian refocusing (at that point)? 
\end{question}

\begin{question}\label{quest: Does every refocusing spacetime admit a Legendrian refocusing metric}
Does every refocusing globally hyperbolic spacetime admit a globally hyperbolic Legendrian refocusing metric? In this case every such spacetime would also admit a strongly refocusing metric by \theoref{thm:Legendrian refocusing implies existence of strongly refocusing metric}.
\end{question}

One way to resolve \questionref{quest: Is every refocusing spacetime Legendrian refocusing} would be to prove that, for the space $\Sigma$ of skies of a globally hyperbolic spacetime, $\Sigma_c = \Sigma_f$, i.e. that Low's reconstructive topology on the space of skies agrees with the topology on $\Sigma$ seen as a subspace of the space $\mathcal{L}$ of Legendrians isotopic to sky in the space of lightrays $\mathcal{N}$ with $\mathcal{C}^{\infty}$ topology. This seems difficult to prove for the following reason: It is easy to produce sequences of Legendrians in $\mathcal{L}$ which converge to a Legendrian in $\mathcal{L}$ in the Hausdorff sense without converging in the $\mathcal{C}^{\infty}$ sense. The difficulty of \questionref{quest: Is every refocusing spacetime Legendrian refocusing} lies in the fact that it is not about $\mathcal{L}$, but about the much smaller and more constrained space $\Sigma$ inside of it.

\section*{Acknowledgments}
The author thanks Stefan Suhr for pointing him to the papers of Gluck and Singer \cite{GluckSinger1978,GluckSinger1979}, which inspired the construction in \secref{sec:Legendrian refocusing spacetimes admit strongly refocusing metrics}. The author is also grateful to Vladimir Chernov for many helpful discussions. 

\section*{Declaration of generative AI and AI-assisted technologies}
During preparation of the manuscript, the author used generative AI tools to help identify possible errors, including in proofs and calculations.

\appendix
\setcounter{thm}{0}
\renewcommand{\thesection}{A}

\section{Topological consequences of refocusing and strong refocusing}\label{appendix: Topological consequences of refocusing and strong refocusing}
There are various topological consequences of (or, equivalently, obstructions to) refocusing phenomena. In this section we give a brief overview. Low proved that any Cauchy surface of a globally hyperbolic refocusing spacetime must be compact \cite{Low2006}. Chernov and Rudyak then proved that the semi-Riemannian universal cover of a globally hyperbolic refocusing spacetime of dimension at least $3$ is again globally hyperbolic and refocusing \cite[Theorem~11.5]{ChernovRudyak2008}. Combining this with Low's theorem gives the following.

\begin{prop}\label{prop:refocusing spacetimes have compact Cauchy surface with finite fundamental group}
Let $(X,g)$ be a refocusing, globally hyperbolic spacetime of dimension at least $3$. Then any Cauchy surface of $(X,g)$ is compact with finite fundamental group.
\end{prop}

In the strongly refocusing case one obtains information about the cohomology as well. The author proved the following Bott-Samelson type result for strongly refocusing spacetimes \cite[Proposition~3.13]{Bauermeister2026}.

\begin{prop}\label{prop:BottSamelsonForStronglyRefocusingSpacetime}
Let $(X,g)$ be a globally hyperbolic, strongly refocusing spacetime of dimension at least $3$. Then any Cauchy surface of $(X,g)$ is compact with finite fundamental group, and its universal cover has the integral cohomology ring of a CROSS.
\end{prop}
The result on cohomology is chiefly a consequence of a more general contact-theoretic statement proved by Frauenfelder, Labrousse, and Schlenk \cite[Theorem~1.13]{FrauenfelderLabrousseSchlenk2015}. That the contact-theoretic result implied the Lorentzian result was known to Chernov and Nemirovski but never published by them.

For refocusing spacetimes, if the analogous cohomological statement holds remains an open question. 
\begin{question}\label{quest: do Cauchy surfaces of refocusing spacetimes have universal cover with cohomology of a CROSS?}
For every globally hyperbolic, refocusing spacetime of dimension at least $3$, does any Cauchy surface of its universal cover have the integral cohomology ring of a CROSS?
\end{question}
This question was also posed in talks given by Chernov and Rudyak. One way to resolve \questionref{quest: do Cauchy surfaces of refocusing spacetimes have universal cover with cohomology of a CROSS?} would be to answer \questionref{quest: Does every refocusing spacetime admit a Legendrian refocusing metric} affirmatively. Conversely, an affirmative answer to \questionref{quest: do Cauchy surfaces of refocusing spacetimes have universal cover with cohomology of a CROSS?} might be an instrumental step in resolving \questionref{quest: Does every refocusing spacetime admit a Legendrian refocusing metric}.

\begin{defin}[$Y^x_l$ manifold]\label{def:Yxl}
Let $(M,h)$ be a complete Riemannian manifold, $x\in M$, and $l>0$. We call $(M,h)$ a \textit{$Y^x_l$ manifold} if $\alpha(l)=x$ for every unit-speed geodesic $\alpha$ with $\alpha(0)=x$.
\end{defin}

The notation is due to Besse \cite[Definitions~7.7]{Besse}. In \cite[Definition 2.2]{ChernovKinlawSadykov2010}, Chernov, Kinlaw and Sadykov introduced the following weakening of \defref{def:Yxl}.

\begin{defin}[$\tilde{Y}^x$ manifold]\label{def:tildeY}
Let $(M,h)$ be a complete Riemannian manifold and $x\in M$. We call $(M,h)$ a \textit{$\tilde{Y}^x$ manifold} if there is $\bar\varepsilon>0$ such that for every $0<\varepsilon<\bar\varepsilon$ there exist $y\in M$ and $l>\varepsilon$ with $d(x,\alpha(l))<\varepsilon$ for every unit-speed geodesic $\alpha$ with $\alpha(0)=y$.
\end{defin}

It is obvious that every $Y^x_l$ manifold is a $\tilde{Y}^x$ manifold. The converse is an open question.

\begin{question}\label{quest:Are there tilde(Y)^x_l manifolds which are not Y^x_l manifolds?}
Are there $\tilde{Y}^x$ manifolds which are not $Y^x_l$ manifolds for any $l>0$?
\end{question}

\begin{defin}[spacetime associated to a Riemannian manifold]\label{def:associated}
The \textit{spacetime associated to} a Riemannian manifold $(M,h)$ is $(X,g)=(M\times\R,\,h-dt^2)$, time-oriented by declaring $\partial_t$ to be future-pointing.
\end{defin}

If $(M,h)$ is complete, then $(X,g)$ is geodesically complete and globally hyperbolic, and each slice $M\times\{t\}$ is a smooth spacelike Cauchy surface. The associated spacetime to a $Y^x_l$ manifold is strongly refocusing with respect to $(x,0)$ and $(x,l)$, and the associated spacetime to a $\tilde{Y}^x$ manifold is refocusing \cite{ChernovKinlawSadykov2010}. Applying \propref{prop:refocusing spacetimes have compact Cauchy surface with finite fundamental group} and applying either \propref{prop:BottSamelsonForStronglyRefocusingSpacetime} to the associated spacetime or \cite[Theorem~1.13]{FrauenfelderLabrousseSchlenk2015} directly, gives the following results.

\begin{thm}
Let $(M,h)$ be a $\tilde{Y}^x$ manifold of dimension at least $2$. Then $M$ is compact with finite fundamental group. If $(M,h)$ is a $Y^x_l$ manifold then its universal cover has the integral cohomology ring of a CROSS.
\end{thm}

Bott and Samelson \cite{Bott1954,Samelson1963} and Bérard-Bergery \cite{BerardBergery1977} proved slightly weaker versions of the cohomological result many decades earlier.


\end{document}